\documentclass{amsart}          
\usepackage[b5paper]{geometry}	\usepackage{amsfonts,amsmath,latexsym,amssymb}
\usepackage{amsthm} 
\usepackage{mathrsfs,upref}

\usepackage{mathptmx}		

\def\<{\langle}
\def\>{\rangle}

\def\BB{{\mathcal B}}

\def\LL{{\mathcal L}}

\def\TT{{\mathcal T}}

\def\bbC{\mathbb{C}}

\def\bbD{\mathbb{D}}
\def\bbT{\mathbb{T}}

\newcommand{\Ksp}[1]{{K^2_{#1}}}
\newcommand{\Ku}{{\Ksp{u}}}

\renewcommand{\implies}{\Rightarrow}

\newcommand{\To}{A^u}

\newcommand{\1}{\mathbf{1}}

\newtheorem{lemma}{Lemma}[section]
\newtheorem{theorem}[lemma]{Theorem}
\newtheorem{corollary}[lemma]{Corollary}

\newtheorem*{examples*}{Examples}

\begin{document}
\title{Commutation relations for truncated Toeplitz operators}
\date{}
\author{Isabelle Chalendar and Dan Timotin} 

\address{Isabelle Chalendar, Universit\'e de Lyon, CNRS UMR 5208, Universit\'e Lyon 1, Institut Camille Jordan, 43 blvd. du 11 Novembre 1918, F-69622 Villeurbanne Cedex,
   France}
   \email{chalendar@math.univ-lyon1.fr}

\address{Dan Timotin, Institute of Mathematics Simion Stoilow of the Romanian Academy, P.O. Box 1-764, Bucharest 014700,
   Romania}
   \email{Dan.Timotin@imar.ro}



\begin{abstract}
For truncated Toeplitz operators, which are compressions of multiplication operators to model subspaces of the Hardy space $H^2$, we obtain criteria for commutation relations. The results show an analogy to the case of Toeplitz matrices, and they extend the theory of Sedlock algebras.
\end{abstract}

\maketitle

\section{Introduction}

Truncated Toeplitz operators are compressions of multiplication operators to model subspaces of the Hardy space $H^2$; they represent a far reaching generalization of classical Toeplitz matrices.  Although particular case had appeared before in the literature, the general theory has been initiated in the seminal paper~\cite{Sa}. Since then, truncated Toeplitz operators have constituted  an active area of research. We mention only a few  relevant papers:~\cite{BBK,BCFMT,CGRW,CRW,GRW,STZ}; see also the recent  survey~\cite{GR} and the references within.

In particular, in~\cite{Se} Sedlock has investigated when a product of truncated Toeplitz operators is itself a truncated Toeplitz operator. It turns out that this does not happen very often. More precisely, there exists a family of classes $\BB^\alpha_u$ (precise definitions in the next section), where $\alpha$ is in the extended complex plane, such that, whenever the product of two nonscalar truncated Toeplitz operators is itself a truncated Toeplitz operator, both operators have to belong to the same class $\BB^\alpha_u$. These classes are commutative algebras, and they are the maximal subalgebras of the subspace of truncated Toeplitz operators.

On the other hand, truncated Toeplitz operators represent a far reaching generalization of classical Toeplitz matrices. Toeplitz matrices whose product is also a Toeplitz matrix are sometimes called generalized circulants~\cite{Da}, and a discussion of the classes $\BB^\alpha_u$ for this particular case appears in~\cite{Sh}. A uniform procedure for imposing conditions on products of Toeplitz matrices has been devised in~\cite{GP}, leading to characterizations of different classes of Toeplitz matrices: normal, unitary, commuting, etc.

The purpose of the present paper is to adapt the approach in~\cite{GP} to the general case of truncated Toeplitz operators on an arbitrary model space. The algebraic relations carry through neatly if we take advantage of a certain unitary operator between different model spaces, called the Crofoot transform. As a consequence, we obtain complete  characterizations of some classes of truncated Toeplitz operators defined by commutation relations.

The plan of the paper is the following. After a preliminary section, we introduce  the Sedlock classes in Section 3 and  the Crofoot transform in Section 4. Section 5 is dedicated to the key technical argument, which is analogous to the one in~\cite{GP}. The main results are then proved in Section 6.

\section{Preliminaries}
 
Our notations are mostly standard: $\bbC$  is the complex plane,  $\bbD=\{z\in\bbC:|z|<1\}$ the unit disc, and  $\bbT=\{z\in\bbC:|z|<1\}$ the unit circle. By $\hat\bbC$ we will denote the extended complex plane $\bbC\cup\{\infty\}$.
As is customary, we will view the Hardy space $H^2$ on $
\bbD$ as a subspace of $L^2(\bbT)$ by identifying functions analytic in $\bbD$ with their radial limits (almost everywhere). Similarly,   the algebra $H^\infty$ of bounded analytic functions in $\bbD$ may be viewed as a closed subalgebra of $L^\infty(\bbT)$.

An inner function $u\in H^\infty$ is characterized by $|u|=1$ almost everywhere on $\bbT$. If $u$ is an inner function and $a\in\bbD$, we define the inner function $u_a$ by
 \[
 u_a(z)=\frac{u(z)-a}{1-\bar a u(z)}.
 \]
 
 If $u$ is an inner function, the model space $\Ku$ is defined by $\Ku=H^2\ominus uH^2$. We denote by $P_{\Ku}$ the orthogonal projection (in $L^2(\bbT)$) onto $\Ku$.

 The conjugation of $L^2(\bbT)$ defined by $\tilde f=u\bar z\bar f $ bijectively maps $\Ku$ to itself; it is this latter restriction that will appear in the sequel.
The space $\Ku$ is a reproducing kernel space of analytic functions on $\bbD$, and the reproducing kernels for points $\lambda\in\bbD$ are 
 \[
 k^u_\lambda(z)=\frac{1-\overline{u(\lambda)}u(z)}{1-\bar\lambda z}.
 \]
The conjugate kernels $\tilde k^u_\lambda$ will also appear; an easy computation yields
 \[
 \tilde k^u_\lambda(z)=\frac{u(z)-u(\lambda)}{z-\lambda}.
 \]
 As shown in~\cite{AC}, in special cases one may have ``reproducing kernels'' for points $\zeta\in\bbT$. Namely, all functions in $\Ku$ have a nontangential limit $f(\zeta)$ in $\zeta\in\bbT$ precisely when $u$ has an angular derivative in the sense of Caratheodory in $\zeta$. In this case the function 
 \[
 k^u_\zeta(z)=\frac{1-\overline{u(\zeta)}u(z)}{1-\bar\zeta z}
 \]
belongs to $\Ku$, and $f(\zeta)=\<f,k_\zeta\>$ for $f\in \Ku$.
  
The truncated Toeplitz operators (TTO) are defined as follows. Note first that, since the reproducing kernels are bounded functions, $\Ku\cap H^\infty$ is dense in $\Ku$.  If $\phi\in L^2(\bbT)$, we consider the map $f\mapsto P_{\Ku}\phi f$ defined on $\Ku\cap H^\infty$. If this map extends to a bounded operator on $\Ku$, we denote it $\To_\phi$ and call it a truncated Toeplitz operator with symbol $\phi$. The set of all TTOs on $\Ku$ is a weakly closed subspace of $\LL(\Ku)$, that we will denote by $\TT_u$.

Truncated Toeplitz operators are closer to Toeplitz matrices than to Toeplitz operators. To start with, the symbol of a TTO is not uniquely defined; it is proved in~\cite{Sa} that $\To_\phi=0$ if and only if $\phi\in uH^2+\overline{uH^2}$. It would be tempting to speak about the uniquely defined \emph{reduced symbol} of a TTO $\To_\phi$ as the projection of $\phi$ onto $L^2\ominus (uH^2+\overline{uH^2}) $. This space can also be written as $(\Ku+\overline{\Ku}) \ominus \bbC(k^u_0-\bar k^u_0)$ (see~\cite{Sa,Se}); in particular, any TTO has a symbol in $\Ku+\overline{\Ku}$. Obviously things simplify when $k^u_0=\bar k^u_0$, which is equivalent to $u(0)=0$; we will have more to say  about this in Section~\ref{se:basic u(0)=0}.

It has been shown in~\cite[Theorem 4.1]{Sa} that TTOs may be characterized algebraically among operators on $\Ku$; the result is the following.
  
 \begin{lemma}\label{le:characterization of TTOs}
  The bounded operator $A$ on $\Ku$ belongs to $\TT_u$ if and only if there are functions $\psi,\chi\in\Ku$ such that
  \[
   \Delta(A):=A-S_uA S_u^* =(\psi\otimes k^u_0)+ (k^u_0\otimes \chi),
  \]
  in which case $A=\To_{\psi+\bar\chi}$.
 \end{lemma}
 
 \begin{examples*}
 
 \begin{enumerate}
  \item If  $\phi(z)=z$, then $\To_\phi$ is the \emph{model operator}~\cite{Ni,SN} on the space $\Ku$; it will be denoted by $S_u$.
  
  \item In~\cite{Sa} are identified all rank one operators in $\TT_u$: they are multiples of $k^u_\lambda\otimes\tilde k^u_\lambda$ and of their adjoints 
 $\tilde k^u_\lambda\otimes k^u_\lambda$, to which are added multiples of $k^u_\zeta\otimes k^u_\zeta$ whenever $u$ has an angular derivative in the sense of Caratheodory in $\zeta\in\bbT$.
 \item  For $\alpha\in\bbD$ the \emph{modified compressed shift}s are defined by
\[
S_u^\alpha=S_u+\frac{\alpha}{1-\alpha\overline{u(0)}}\, k^u_0\otimes\tilde k^u_0.
\]

If $\alpha\in\bbD$, then $S_u^\alpha$ is unitarily equivalent to $S_{u_\alpha}$, and is thus
a completely non-unitary contraction (whose characteristic function, in the sense of Sz.Nagy--Foias~\cite{SN}, is  $u_\alpha$). If $\alpha\in\bbT$, then $S_u^\alpha$ is unitary, with singular spectral measure and multiplicity one (these are precisely the Clark unitary operators defined in~\cite{Cl}).
 \end{enumerate}

\end{examples*}

\section{Sedlock classes}\label{se:sedlock}

 The Sedlock classes  $\BB_u^\alpha\subset \TT_u$, with $\alpha\in\hat\bbC$, have been introduced in~\cite{Se} 
 in connection to multiplication properties of TTOs.  For $\alpha\in\bbC$,  
  $\BB_u^\alpha$ is  the set of operators in $\TT_u$ which have a symbol of the form $\phi+\alpha\overline{S_u\tilde\phi}+c$, where $\phi\in\Ku$ and $c\in\bbC$; while, for $\alpha=\infty$, $\BB_u^\infty$ is the set of TTOs which have an antiholomorphic symbol. The following are the main results proved in~\cite{Se}.
 
 \begin{theorem}\label{th:sedlock}
  {\rm(i)} For any $\alpha\in\hat\bbC$, $\BB_u^\alpha$ is a commutative weakly closed algebra. 
  
    {\rm(ii)} If $\alpha\not=\alpha'$, then $\BB_u^\alpha\cap\BB_u^{\alpha'}=\bbC I $.
  
    {\rm(iii)} $A\in \BB_u^\alpha$ if and only if $A^*\in \BB_u^{1/\bar\alpha}$.
  
    {\rm(iv)} If $\alpha\in\overline{\bbD}$, then $\BB_u^\alpha=\{ S_u^\alpha \}'$ (the commutant of $S_u^\alpha$).
  
    {\rm(v)} If $A,B\in \TT_u$, then $AB\in\TT_u$ if and only if either one of the operators is a scalar, or both belong to the same class $\BB_u^\alpha$ for some $\alpha\in\hat\bbC$. In the last case we also have  $AB\in\BB_u^\alpha$.
  
    {\rm(vi)} The classes $\BB_u^\alpha$ are precisely the maximal subalgebras of $\TT_u$.
 \end{theorem}

As the algebras $\BB^\alpha_u$ are the commutants of modified compressed shifts, they may be given a more concrete description. This is done in~\cite[Section 6]{Se}, and we present below a brief summary of the results therein. There are basically two distinct types of  Sedlock classes, depending on whether $|\alpha|=1$ or not, and the case $|\alpha|>1$ is reduced to $|\alpha|<1$ by taking adjoints.

\begin{enumerate}
\item
If $|\alpha|=1$, then $S^\alpha_u$ is a unitary operator  of multiplicity one, with singular spectral measure $\mu_\alpha$. Thus $\BB^\alpha_u=\{S^\alpha_u\}'$
is a maximal abelian subalgebra of $\LL(\Ku)$, and its elements may be described as functions $\Phi(S^\alpha_u)$ with $\Phi\in L^\infty(\mu_\alpha)$.

\item
 If $|\alpha|\not=1$, suppose first that $|\alpha|<1$. Then $S^\alpha_u$ is a completely nonunitary contraction, that has a functional calculus with functions in $H^\infty$~\cite{SN}. Its commutant $\BB^\alpha_u$ is a weakly closed nonselfadjoint algebra; its elements are the functions $\Psi(S^\alpha_u)$ with $\Psi\in H^\infty$, and we may identify their symbols as TTOs by the formula 
 \[
 \Psi(S^\alpha_u)=\To_{\frac{\Psi}{1-\alpha \bar u}}.
 \]
 
If $|\alpha|>1$, then $S^{1/\bar\alpha}_u$ is a completely nonunitary contraction, and using Theorem~\ref{th:sedlock} (iii) the elements of $\BB^\alpha_u$ may be described as
\[
\Psi(S^{1/\bar\alpha}_u)^*=\To_{\frac{\alpha\bar\phi}{\alpha-u}}
\]
for $ \Psi\in H^\infty $.
\end{enumerate}

It is worth mentioning the following simple corollary, which determines when the product of two TTOs is zero.

\begin{corollary}\label{co:product=0}
If $\To_\phi, \To_\psi$ are nonzero operators in $\TT_u$ and $\To_\phi \To_\psi=0$, then there is $\alpha\in\hat\bbC$ such that $\To_\phi, \To_\psi\in \BB^\alpha_u$. Moreover:
\begin{enumerate}
\item
If $|\alpha|=1$, then $\To_\phi=\Phi(S^\alpha_u)$, $\To_\psi=\Psi(S^\alpha_u)$, with $\Phi,\Psi\in L^\infty(\mu_\alpha)$ and $\Phi\Psi=0$ $\mu_\alpha$-almost everywhere.

\item
If $|\alpha|<1$, then $\To_\phi=\Phi(S^\alpha_u)$, $\To_\psi=\Psi(S^\alpha_u)$, with $\Phi,\Psi\in H^\infty$, and the inner function $u_\alpha$ divides $\Phi\Psi$.

\item
If $|\alpha|>1$, then $\To_\phi=\Phi(S^{1/\bar\alpha}_u)^*$, $\To_\psi=\Psi(S^{1/\bar\alpha}_u)^*$, with $\Phi,\Psi\in H^\infty$, and the inner function $u_{1/\bar\alpha}=\frac{1-\bar\alpha u}{u-\alpha}$ divides $\Phi\Psi$.
\end{enumerate}
\end{corollary}

\begin{proof}
Most of the statements are immediate consequences of the remarks above. For point (ii), one should note that if $h\in H^\infty$ and $h(S^\alpha_u)=0$, then $u_\alpha$ divides $h$. This is proved directly in~\cite[Section 6]{Se}; alternately, it  follows from the fact, noted above, that the characteristic function of $S^\alpha_u$ is 
$u_\alpha$.
\end{proof}
 
 We end this section with a continuity property of Sedlock classes.
 
 \begin{lemma}\label{le:sedlock continuity}
Suppose $\alpha_n, \alpha\in\bbC$, $\alpha_n\to \alpha$, $A_n\in\BB_u^{\alpha_n}$, and $A_n\to A$. Then $A\in\BB_u^\alpha$.
\end{lemma}

\begin{proof}
For $\alpha_n, \alpha\in \bar\bbD$ the result follows from Theorem~\ref{th:sedlock} (iv), once we note that $\alpha_n\to \alpha$ implies $S_u^{\alpha_n}\to S_u^\alpha$. If $\alpha\not\in\bbD$, we use Theorem~\ref{th:sedlock} (iii) to reduce it to the previous case.
\end{proof}
 
In the sequel we will usually assume that $\alpha\in\bbC$; the obvious modifications of the arguments required when $\alpha=\infty$ are left to the reader. 
 
 \section{The Crofoot transform}
 
Let $u$ be an inner function and $a\in\bbD$.  The Crofoot transform $J=J(u,a)$ is the unitary operator $J:\Ku\to\Ksp{u_a}$ defined by
 \[
  J(f)=\frac{\sqrt{1-|a|^2}}{1-\bar a u} f.
 \]
It is proved in~\cite[Theorem 13.2]{Sa} that 
\begin{equation}\label{eq:crofoot1}
 J\TT_u J^*=\TT_{u_a}.
\end{equation}

The next result could be obtained by tedious calculations, but we prefer  a shorter argument based on the  previous section. 

\begin{theorem}\label{le:crofoot}
 If $\alpha\in\hat\bbC$, then  $J\BB_u^\alpha J^*=\BB_u^{\beta}$, where $\beta=\frac{\alpha-a}{1-\bar a\alpha}$.
\end{theorem}

\begin{proof}
Since $\BB_u^\alpha$ is a maximal subalgebra of $\TT_u$, it follows from~\eqref{eq:crofoot1} that $J\BB_u^\alpha J^*$ is a maximal algebra of $\TT_{u_a}$, and thus, by Theorem~\ref{th:sedlock}, it must be equal to $\BB_{u_a}^\beta$ for some $\beta\in\hat\bbC$. To obtain the precise value of $\beta$, it is enough to look at the Crofoot transform of a single nonscalar operator; this we will do in the sequel. We may assume that $\dim \Ku>1$, since otherwise there is nothing to prove.

Suppose first that $|\alpha|<1$. It is shown in~\cite[Example 5.3]{Se} that for any $\lambda\in\bbD$ the rank one operator $\tilde k^u_\lambda\otimes k^u_\lambda$ belongs to $\BB_u^{u(\lambda)}$; also, $\tilde k^u_\lambda\otimes k^u_\lambda$ is not scalar since $\dim\Ku>1$.

If $f\in\Ksp{u}$, then
\[
f(\lambda)=
\frac{1-\bar a u(\lambda)}{\sqrt{1-|a|^2}} (Jf)(\lambda)=
 \<Jf,\frac{1- a \overline{u(\lambda)}}{\sqrt{1-|a|^2}} k^{u_a}_\lambda\>
=\<f,\frac{1- a \overline{u(\lambda)}}{\sqrt{1-|a|^2}}J^* k^{u_a}_\lambda\>.
\]
Therefore $J^* k^{u_a}_\lambda$ is a multiple of $k^u_\lambda$, or, equivalently, $ k^{u_a}_\lambda$ is a multiple of $Jk^u_\lambda$. Since $J$ commutes with the respective conjugations on $\Ksp{u}$ and $\Ksp{u_a}$, the conjugate kernel $ \tilde k^{u_a}_\lambda$ is a multiple of $J\tilde k^u_\lambda$. Therefore $J(\tilde k^u_\lambda\otimes k^u_\lambda)J^*$ is a multiple of $\tilde k^{u_a}_\lambda\otimes k^{u_a}_\lambda$, and thus belongs to $\BB_{u_a}^{{u_a}(\lambda)}$. Since $u_a(\lambda)=\frac{u(\lambda)-a}{1-\bar a u(\lambda)}$, we have found, in the case $\alpha= u(\lambda)$,  a nonscalar operator in the class $\BB_u^\alpha$ whose Crofoot transform is in  $\BB_{u_a}^\beta$, with $\beta=\frac{\alpha-a}{1-\bar a \alpha}$. By Theorem~\ref{th:sedlock} (ii) the same must then be true for the whole class.

The result is thus proved for points  in  $u(\bbD)$; since $u$ is inner, this is a dense set in $\bar\bbD$ (see, for instance,~\cite[Theorem 6.6]{Ga}). For $\alpha\in\bar\bbD$ outside this set, choose some $w\in\bbD$ such that, if $\phi=k^u_w+\alpha\overline{S_u\tilde k^u_w}$, then $\To_\phi$ is not scalar. Take a sequence $\alpha_n\to\alpha$, $\alpha_n\in u(\bbD)$; then $\phi_n=k^u_w+\alpha_n\overline{S_u\tilde k^u_w}$ tend uniformly to $\phi$, and therefore $\To_{\phi_n}\to \To_\phi$, $J\To_{\phi_n}J^*\to J\To_\phi J^*$. We have  $\To_{\phi_n}\in \BB_u^{\alpha_n}$ and $\To_\phi\in \BB_u^\alpha$ by the definition of the Sedlock classes. Since $\alpha_n\in u(\bbD)$, $J\To_{\phi_n}J^*\in \BB_{u_a}^{\beta_n}$, with $\beta_n:=\frac{\alpha_n-a}{1-\bar a \alpha_n}\to \beta:=\frac{\alpha-a}{1-\bar a \alpha} $. Applying Lemma~\ref{le:sedlock continuity}, it follows that $J\To_\phi J^*\in \BB_{u_a}^\beta$. 
So again we have found a nonscalar operator in $\BB_u^\alpha$, whose Crofoot transform is in $\BB_{u_a}^\beta$ with $\beta=\frac{\alpha-a}{1-\bar a \alpha}$, and by Theorem~\ref{th:sedlock} (ii) the same must be true for the whole class.

Finally, if $|\alpha|>1$, then $\alpha'=1/\bar\alpha\in\bbD$, and, if $\beta'=\frac{\alpha'-a}{1-\bar a \alpha'}$, then $1/\overline{\beta'}=\beta$. Therefore, using the result already proved for $\alpha'$ and  Theorem~\ref{th:sedlock} (iii), we obtain
\[
J\BB^\alpha_u J^*= J(\BB^{\alpha'}_u)^*J^*
=\big(J\BB^{\alpha'}_uJ^*\big)^* = (\BB^{\beta'}_{u_a})^*=\BB^\beta_{u_a},
\]
thus ending the proof of the theorem.
\end{proof}

Note that the particular case $a=\alpha$ appears in~\cite[Section 6]{Se}.
We will only use  the Crofoot transform obtained by taking $a=u(0)$; in this case $u_a(0)=0$.

 \section{Basic commutation formulas}\label{se:basic u(0)=0}

 In this section  the inner function $u$ is subjected to the condition $u(0)=0$. Then $u=zu_1$, $k^u_0=\1$ (the constant function equal to 1), and $\tilde  k^u_0=u_1$; also, we have the direct sum decompositions
\begin{align}
 \Ku&=\bbC \1\oplus z\Ksp{u_1},\label{eq:decomposition1}\\ (u H^2 +\overline{u H^2})^\perp&=
 \overline{\Ku}+\Ku=
 \overline{ z\Ksp{u_1}}\oplus \bbC \1\oplus z\Ksp{u_1}.\label{eq:decomposition2}
\end{align}

 Any TTO has a unique symbol $\phi\in (u H^2 +\overline{u H^2})^\perp$, and according to~\eqref{eq:decomposition2} we may write 
 \begin{equation}\label{eq:decomposition of phi}
  \phi=\phi_++\bar \phi_-+\phi_0
 \end{equation}
with $\phi_\pm\in z\Ksp{u_1}$ and $\phi_0\in\bbC$. Whenever $u(0)=0$,  the operator $\To_\phi$ will have the symbol  $\phi$ in $\overline{\Ku}+\Ku$, and we will consistently use  the decomposition~\eqref{eq:decomposition of phi}. Note that $(\To_\phi)^*=\To_{\bar \phi}$, and $(\bar \phi)_\pm=\phi_\mp$, $(\bar \phi)_0=\overline{\phi_0}$.

We define a conjugation $\breve{}$ on $z\Ksp{u_1}$, that we will call the \emph{reduced conjugation}, by transporting the conjugation on $\Ksp{u_1}$; that is, for $f\in z\Ksp{u_1}$,
\begin{equation}\label{eq:def of involution}
 \breve f=z\bar f u_1.
\end{equation}

The Sedlock classes can be easily identified in terms of $\phi_\pm$; namely, $\To_\phi\in\BB_u^\alpha$ if and only if $\breve\phi_-=\alpha \phi_+$.

Finally, let us note the formulas
\begin{equation}\label{eq:defect formulas}
 \Delta(I)=I-S_uS_u^*=\1\otimes\1, \qquad I-S_u^*S_u= u_1\otimes u_1.
\end{equation}

The next is the correspondent of~\cite[Lemma 2.3]{GP}.

 \begin{lemma}\label{le:basic lemma} Suppose $u(0)=0$.
If $\To_\phi, \To_\psi\in\TT_u$, then
\[
\begin{split}
 \Delta(\To_\phi\To_\psi)&= \phi_+\otimes\psi_- - \breve\phi_-\otimes \breve \psi_+\\
 &\qquad+(\To_\phi \psi_++\psi_0 \phi_+ +\phi_0\psi_0\1)\otimes \1 
 +\1\otimes (S_u (\To_\psi)^* S_u^* \phi_-+\bar\phi_0 \psi_-).
\end{split}
\]
 \end{lemma}
 
 \begin{proof}
  Denote $\hat \phi=\phi-\phi(0)$, $\hat \psi=\phi-\psi(0)$. We have 
  \[
   \Delta(\To_\phi\To_\psi)= \Delta(\To_{\hat\phi}\To_{\hat\psi})+\psi_0 \Delta(\To_{\hat\phi}) + \phi_0 \Delta(\To_{\hat\psi}) +\phi_0 \psi_0 (\1\otimes\1).
 \]
By Lemma~\ref{le:characterization of TTOs}, we have
\begin{equation}\label{eq:delta0}
 \Delta(\To_{\hat\phi})=\phi_+\otimes\1 +\1\otimes \phi_-,\qquad 
 \Delta(\To_{\hat\psi})=\psi_+\otimes\1 +\1\otimes \psi_-,
\end{equation}
and therefore
\begin{equation}\label{eq:delta1}
 \Delta(\To_\phi\To_\psi)= \Delta(\To_{\hat\phi}\To_{\hat\psi})
 +(\psi_0\phi_+ +\phi_0 \psi_+ +\phi_0 \psi_0 \1)\otimes \1+
 \1\otimes (\bar\psi_0\phi_-+\bar\phi_0 \psi_-).
 \end{equation}

 Now, using~\eqref{eq:defect formulas} and~\eqref{eq:delta0},
 \begin{align*}
  \Delta(\To_{\hat\phi}\To_{\hat\psi})&= 
  \To_{\hat\phi}\To_{\hat\psi} -S_u\To_{\hat\phi}\To_{\hat\psi}S_u^*\\
 & =
   \To_{\hat\phi}\To_{\hat\psi}-\To_{\hat\phi} S_u \To_{\hat\psi} S_u^*
   + \To_{\hat\phi} S_u \To_{\hat\psi} S_u^* - 
   S_u \To_{\hat\phi} (S_u^*S_u+u_1\otimes u_1) \To_{\hat\psi} S_u^*\\
   &= \To_{\hat\phi} \Delta(\To_{\hat\psi}) +
   \Delta(\To_{\hat\phi}) S_u \To_{\hat\psi} S_u^* 
   -S_u \To_{\hat\phi} (u_1\otimes u_1)  \To_{\hat\psi} S_u^*\\
  & = \To_{\hat\phi} (\psi_+\otimes\1 +\1\otimes \psi_-)+
  ( \phi_+\otimes\1 +\1\otimes \phi_-) S_u \To_{\hat\psi} S_u^* -
  (S_u \To_{\hat\phi}u_1\otimes S_u (\To_{\hat\psi})^* u_1)
  \end{align*}
We have $\To_{\hat\phi}\1=\phi_+$, $S_u^*\1=0$, so the sum of the first two terms on the last line is
\[
 \To_{\hat\phi} \psi_+\otimes \1+ \phi_+\otimes\psi_-+ \1\otimes S_u(\To_{\hat\psi})^* S_u^*\phi_-.
\]
Further, $\To_{\hat\phi}u_1=P_\Ku \hat\phi u_1= P_\Ku \phi_+u_1+ P_\Ku\bar\phi_-u_1$. Since $\phi_+\in z\Ksp{u_1}$, $\phi_+u_1$ has $zu_1=u$ as a factor, and thus is orthogonal to $\Ku$. Also, $\bar\phi_-u_1=\bar z z\bar\phi_-u_1=\bar z\breve\phi_-$, and  $\breve\phi_-\in z \Ksp{u_1}$ implies $\bar z\breve\phi_-\in\Ku$, whence 
$\To_{\hat\phi}u_1= \bar z\breve\phi_-$. Therefore $S_u \To_{\hat\phi}u_1=P_\Ku \breve\phi_-=\breve\phi_-$.

Taking into account the relation $(\To_{\hat \psi})^*=\To_{\bar{\hat \psi}}=\To_{\psi_-+\bar\psi_+}$, a similar computation yields $S_u (\To_{\hat\psi})^*u_1=\breve\psi_+$. Therefore
\begin{equation}\label{eq:delta2}
\Delta(\To_{\hat\phi}\To_{\hat\psi})=
 \To_{\hat\phi} \psi_+\otimes \1+ \phi_+\otimes\psi_-+ \1\otimes S_u(\To_{\hat\psi})^* S_u^*\phi_--
 \breve \phi_-\otimes\breve\psi_+.
\end{equation}
Gathering~\eqref{eq:delta1} and~\eqref{eq:delta2} ends the proof of the lemma.
\end{proof}
 
 From here follows the basic theorem, which corresponds to~\cite[Theorem 3.1]{GP}.
 
 \begin{theorem}\label{th:basic theorem}
  Suppose $u(0)=0$ and $\To_\phi, \To_\psi, \To_\zeta, \To_\eta \in\TT_u$. Then $\To_\phi \To_\psi- \To_\zeta \To_\eta \in\TT_u$ if and only if
  \begin{equation}\label{eq:basic condition}
   \phi_+\otimes \psi_- - \breve \phi_-\otimes \breve\psi_+=
   \zeta_+\otimes\eta_- - \breve \zeta_- \otimes \breve\eta_+.
  \end{equation}
 \end{theorem}

 \begin{proof}
  By Lemma~\ref{le:basic lemma}, there exist $f,g\in \Ku$ such that
  \[
   \Delta(\To_\phi \To_\psi- \To_\zeta \To_\eta)= \phi_+\otimes \psi_- - \breve \phi_-\otimes \breve\psi_+-
   \zeta_+\otimes\eta_- +  \breve \zeta_- \otimes \breve\eta_+ +
   f\otimes\1 +\1\otimes g.
  \]
From Lemma~\ref{le:characterization of TTOs} it follows that $\To_\phi \To_\psi- \To_\zeta \To_\eta \in\TT_u$ if and only if there exist $f_1,g_1\in\Ku$ such that
\begin{equation}\label{eq:intermediate condition}
 \phi_+\otimes \psi_- - \breve \phi_-\otimes \breve\psi_+-
   \zeta_+\otimes\eta_- +  \breve \zeta_- \otimes \breve\eta_+=
    f_1\otimes\1 +\1\otimes g_1.
\end{equation}

Now, if we consider the orthogonal decomposition~\eqref{eq:decomposition1}, we can write operators on $\Ku$ as $2\times 2$ block matrices. With respect to this decomposition, the left hand side of~\eqref{eq:intermediate condition} has zeros on the first row and column, since $\phi_\pm, \psi_\pm, \zeta_\pm, \eta_\pm\in z\Ksp{u_1}$. Meanwhile,
 the right hand side is the general form of an operator that has zeros in the lower right corner. It follows that both sides have to be zero, so, in particular,~\eqref{eq:basic condition} is true.
 \end{proof}

 \section{Main results}
 
 As noticed above, the Crofoot classes have been introduced in connection with multiplication properties of TTOs, and the main result in this direction is Theorem~\ref{th:sedlock} (v). As a consequence, a characterization of unitary TTOs is obtained in~\cite{Se}. In the sequel we use  Theorem~\ref{th:basic theorem} in order to  improve that result (see Theorem~\ref{th:unitary} below), as well as to obtain complete descriptions of other classes of TTOs.
 
 The first result discusses commuting TTOs.

 \begin{theorem}\label{th:commutation}
 Let $u$ be an inner function.
  If $\To_\phi, \To_\psi\in \TT_u$, then the following are equivalent:
  
  {\rm(i)} $\To_\phi \To_\psi=\To_\psi \To_\phi$.
 
 \smallskip
  {\rm(ii)} $\To_\phi \To_\psi-\To_\psi \To_\phi\in \TT_u$.
 
 \smallskip
  {\rm(iii)} One of the following is true:
 \begin{itemize}
  \item[(1)] There exists $\alpha\in\hat\bbC$ such that $\To_\phi$ and $\To_\psi$ both belong to $\BB^\alpha_u$.
  \item[(2)] The operators $I, \To_\phi, \To_\psi$ are not linearly independent.
 \end{itemize}
 \end{theorem}

 \begin{proof}
  It is obvious that (i)$\implies$(ii). For (iii)$\implies$(i), in case (1) commutativity follows from Sedlock's result, while in case (2) one of the TTOs is a linear combination of the identity and the other. So we are left to prove that   (ii)$\implies$(iii).
  
  Both conditions (ii) and (iii) are invariant if we apply a Crofoot transform: since the transform is unitary, this is obvious for (iii)(2). For (ii) it follows from~\eqref{eq:crofoot1}, while for (iii)(1) it is a consequence of Lemma~\ref{le:crofoot}. So we may assume for the rest of the proof that $u(0)=0$, and thus apply the results from Section~\ref{se:basic u(0)=0}.

  Assume then that $\To_\phi \To_\psi-\To_\psi \To_\phi\in \TT_u$. Applying Theorem~\ref{th:basic theorem} with $\eta=\phi$ and $\zeta=\psi$, formula~\eqref{eq:basic condition} becomes
   \begin{equation}\label{eq:basic commutation}
   \phi_+\otimes \psi_- - \breve \phi_-\otimes \breve\psi_+=
   \psi_+\otimes\phi_- - \breve \psi_- \otimes \breve\phi_+.
  \end{equation}
 The operators on the two sides of this equality  have rank at most two. If the rank is at most one, then  $\{\phi_+, \breve\phi_-\}$ and $\{\psi_+, \breve\psi_-\}$ are both pairs of linearly dependent functions. Suppose, for instance, that $\phi_-\not=0$ and $\breve\phi_-=\alpha\phi_+$. Then~\eqref{eq:basic commutation} yields
 \[
  \phi_+\otimes (\psi_-- \bar\alpha\breve\psi_+)=
  (\bar\alpha\psi_+-\breve\psi_-)\otimes\breve\phi_+.
 \]
The equality of the rank one operators implies the existence of $a\in\bbC$ such that 
\[
 \psi_--\bar\alpha\breve\psi_+=a\breve\phi_+,\qquad \bar\alpha\psi_+-\breve\psi_-=\bar a \phi_+.
\]
Applying the reduced conjugation to the first equation and comparing the result to the second, we see that $a=0$. Thus $\breve\psi_-=\alpha\psi_+$, and thus $\To_\phi$ and $\To_\psi$  both belong to  $\BB_u^{\alpha}$; that is, (1) is true.

Suppose now that the rank of the operators in~\eqref{eq:basic commutation} is two. The spaces spanned by $\{\phi_+, \breve\phi_-\}$ and by $\{\psi_+, \breve\psi_-\}$ are equal, and thus there exist $a_{11}, a_{12}, a_{21}, a_{22}\in\bbC$ such that
\[
 \psi_+= a_{11}\phi_+ +a_{12}\breve\phi_-, \qquad \breve\psi_-= a_{21}\phi_+ +a_{22}\breve\phi_-.
\]
Replacing these formulas in~\eqref{eq:basic commutation} yields
\[
 \left[ 2 a_{21}\phi_+ +( a_{22}- a_{11})\breve\phi_-\right] \otimes \breve\phi_+ +
 \left[ ( a_{22}- a_{11})\phi_+ -2 a_{12}  \breve\phi_-\right]\otimes \phi_-=0,
\]
and then the linear independence of $\phi_+$ and $\breve\phi_-$ implies that $a_{12}= a_{21}=0$ and $a_{11}= a_{22}=a$. Thus $\phi_-= a\phi_+$, $\breve\psi_-= a\breve\phi_-$, 
$\psi_-=\bar a\phi_-$,
and
\[
 \To_\psi=\psi_0 I+\To_{\psi_++\bar\psi_-}=\psi_0 I +a\To_{\phi_++\bar\phi_-} 
 =a\To_\phi+(\psi_0-\phi_0)I.
\]
Therefore in this case (2) is satisfied. This ends the proof of the theorem.  
 \end{proof}
 
 One can obtain as a consequence the characterization of normal TTOs.
 
 \begin{theorem}\label{th:normal} 
 Let $u$ be an inner function.
 If $\To_\phi\in\TT_u$, then the following are equivalent:
 
 {\rm(i)} $\To_\phi\in\TT_u$ is normal.
 
 {\rm(ii)} $\To_\phi (\To_\phi)^*-(\To_\phi)^* \To_\phi\in \TT_u$.
 
 {\rm(iii)}
  One of the following is true:
  \begin{itemize}
   \item[(1)] There exists $\alpha\in\bbT$ such that $\To_\phi$ belongs to $\BB^\alpha_u$.
   \item[(2)] $\To_\phi$ is a linear combination of a selfadjoint TTO and the identity.
  \end{itemize}
 \end{theorem}

 \begin{proof}
  By applying Theorem~\ref{th:commutation} to the case $\psi=\bar\phi$, we obtain the equivalence of (i), (ii), and (iii${}'$), where (iii${'}$) states that one of the following is true:
 \begin{itemize}
 \item[$(1')$]  There exists $\alpha\in\bbC$ such that $\To_\phi$ and $(\To_\phi)^*$  both belong to $\BB^\alpha_u$.
 \item[$(2')$]  The operators $I, \To_\phi, (\To_\phi)^*$ are not linearly independent.
 \end{itemize}

If $\To_\phi$ is a multiple of the identity, then (1), (2), $(1')$, $(2')$ are all satisfied. Suppose this is not the case. If $\To_\phi\in\BB^\alpha_u$, then $(\To_\phi)^*\in\BB^{\bar\alpha^{-1}}$. If $(1')$ is true, then we must have $\bar\alpha^{-1}=\alpha$, or $|\alpha|=1$; thus (1) is equivalent to  $(1')$.

If (2) is true, then $\To_\phi=a A+bI$, with $A=A^*$ and $a\not=0$; then $(\To_\phi)^*=\frac{\bar a}{a}\To_\phi+\frac{a\bar b-\bar a b}{a} I$, and thus $(2')$ is true. 
Conversely, suppose $(\To_\phi)^*=c \To_\phi+dI$. Since we have asumed that $T_\phi$ is not a scalar, at least one of $\Re \To_\phi, \Im \To_\phi$ is not a scalar. Say this is $\Re \To_\phi$; then $\Re\To_\phi= (c+1)\To_\phi+dI$, with $c\not=-1$, and thus 
$\To_\phi=(c+1)^{-1}(\Re\To_\phi-dI$); therefore (2) is true.

Thus $(1)\Leftrightarrow(1')$ and $(2)\Leftrightarrow(2')$; this ends the proof of the theorem.
  \end{proof}

It is proved in~\cite{Se} that a TTO $\To_\phi$ is unitary if and only if it belongs to some class $\BB^\alpha_u$ for some $\alpha\in\bbT$. In this case $\To_\phi=\Phi(S^\alpha_u)$, where $|\Phi|=1$ $\mu_\alpha$-almost everywhere. With our method we can obtain a slight improvement of this result.

\begin{theorem}\label{th:unitary}
Let $u$ be an inner function.
 If $\To_\phi\in\TT_u$, then the following are equivalent:
 \begin{enumerate}
  \item $\To_\phi$ is unitary.
  \item $\To_\phi$ is an isometry.
  \item $\To_\phi$ is a coisometry.
  \item $(\To_\phi)^* \To_\phi-I\in \TT_u$.
  \item $\To_\phi (\To_\phi)^*-I\in \TT_u$.
  \item $\To_\phi\in\BB^\alpha_u$ for some $\alpha\in\bbT$, and $\To_\phi=\Phi(S^\alpha_u)$, where $|\Phi|=1$ $\mu_\alpha$-almost everywhere. 
 \end{enumerate}
\end{theorem}

\begin{proof}
 The implications (i)$\implies$(ii), (i)$\implies$(iii), (ii)$\implies$(iv), (iii)$\implies$(v), and 
 (vi)$\implies$(i) are all immediate.

To prove (v)$\implies$(vi), we may assume, as in the proof of Theorem~\ref{th:commutation}, that $u(0)=0$. We may then apply Theorem~\ref{th:basic theorem} to the case $\psi=\bar\phi$, $\zeta=\eta=\1$, which implies $\zeta_\pm=\eta_\pm=0$. We obtain then 
\[
 \phi_+\otimes\phi_+ =\breve\phi_-\otimes\breve\phi_-.
\]
Therefore there exists $\alpha\in\bbT$ such that $\breve\phi_-=\alpha\phi_+$; that is, $\To_\phi\in\BB^\alpha_u$. The particular form of $\To_\phi$ is then a consequence of the description of $\BB^\alpha_u$ in Section~\ref{se:sedlock}.

Finally, if (iv) is true, then (v) is true for $(\To_\phi)^*=\To_{\bar\phi}$. Therefore the previous paragraph yields $\To_{\bar\phi}\in\BB^\alpha_u$ for some $\alpha\in\bbT$, whence $\To_{\bar\phi}\in\BB^\alpha_u$. Thus (iv)$\implies$(vi), which ends the proof of the theorem. 
\end{proof}

In particular, there do not exist nonunitary isometries or coisometries in $\TT_u$. This can also be obtained as a consequence of the complex symmetry of  the truncated Toeplitz operators with respect to the conjugation on $\Ku$ (see~\cite{GPu}).
 
\section*{Acknowledgements}

The first author was partially supported by the ANR project ANR-09-BLAN-0058-01. The second author was partially supported by a grant of the Romanian National Authority for Scientific
Research, CNCS Ð UEFISCDI, project number PN-II-ID-PCE-2011-3-0119.

 \renewcommand{\bibname}{\sc}


\begin{thebibliography}{99}
  
 %

  
  \bibitem{AC}
  {\bibname P.R. Ahern and D.N. Clark}, `Radial limits and invariant subspaces',
{\em Amer. J. Math. }92 (1970), 332--342. 
  
  \bibitem{BBK} {\bibname A. Baranov, R. Bessonov, and V. Kapustin}, `Symbols of truncated Toeplitz operators', 
{\em J. Funct. Anal. }261 (2011),  3437--3456.
  
  \bibitem{BCFMT} {\bibname A. Baranov, I. Chalendar, E. Fricain, J. Mashreghi, and D. Timotin}, `Bounded symbols and reproducing kernel thesis for truncated Toeplitz operators', {\em J. Funct. Anal. }259 (2010), 2673--2701.
  
  \bibitem{CGRW} {\bibname J.A. Cima, S.R. Garcia, W.T. Ross, and W.R.Wogen}, `Truncated Toeplitz operators: spatial isomorphism, unitary equivalence, and similarity', {\em Indiana Univ. Math. J.  }59  (2010),  595--620.
  
  \bibitem{CRW} {\bibname J.A. Cima, W.T. Ross, and W.R. Wogen}, `Truncated Toeplitz operators on finite dimensional spaces',
  {\em Oper. Matrices  }2  (2008),  357--369.
  
  
  \bibitem{Cl} D.N. Clark: One dimensional perturbations of restricted shifts, \emph{J. Analyse Math.} {\bf 25} (1972), 169--191.

  
  \bibitem{Da} {\bibname P.J. Davis}, {\em Circulant Matrices} (Wiley, New York, 1979).
  
 \bibitem{GPu}
 {\bibname S.R. Garcia and M. Putinar}, `Complex symmetric operators and applications', {\em Trans. Amer. Math. Soc. }358  (2006),  1285--1315.
 
 \bibitem{GR}
 {\bibname S.R. Garcia and W.T. Ross}, `Recent progress on truncated Toeplitz operators', preprint arXiv:1108.1858v4 [math.CV].
  
  \bibitem{GRW} {\bibname S.R. Garcia, W.T. Ross, and W.R. Wogen}, 
  `Spatial isomorphisms of algebras of truncated Toeplitz operators', {\em Indiana Univ. Math. J.  }59  (2010),   1971--2000.
  
  \bibitem{Ga} {\bibname J.B. Garnett}, {\em Bounded Analytic Functions} (Springer, New York, 2007).
  
  \bibitem{GP} {\bibname C. Gu and L. Patton}, `Commutation relations for Toeplitz and Hankel matrices',  
{\em  SIAM J. Matrix Anal. Appl. }24  (2003),  728--746.
  
  \bibitem{Ni}
  {\bibname N.K. Nikolski}, {\em Operators, functions, and systems: an easy reading} (American Mathematical Society, Providence, RI, 2002).
  
  \bibitem{Sa} {\bibname D. Sarason}, `Algebraic properties of truncated Toeplitz operators',  
{\em Oper. Matrices }1  (2007),  491--526.
  
  \bibitem{Se} {\bibname N.A. Sedlock}, `Algebras of truncated Toeplitz operators', {\em Oper. Matrices  }5  (2011),  309--326.
  
  \bibitem{Sh} {\bibname T. Shalom}, `On algebras of Toeplitz matrices', {\em Linear Algebra Appl.  }96 (1987), 211--226.
  
  \bibitem{STZ} {\bibname E. Strouse, D. Timotin, and M. Zarrabi}, `Unitary equivalence to truncated Toeplitz operators', {\em Indiana Univ. Math. J.}, to appear.
  
  \bibitem{SN}
  {\bibname B. Sz.-Nagy and C. Foias}, {\em Harmonic analysis of operators on Hilbert space} (North-Holland Publishing Co., Amsterdam-London, 1970).
  
 \end{thebibliography}
 \end{document}